\documentclass[10 pt,a4paper,twoside,reqno]{amsart}
\usepackage{amsfonts,amssymb,amscd,amsmath,enumerate,verbatim,calc}

\renewcommand{\baselinestretch}{1.2}
\newtheorem{TH22}{Theorem}
\newtheorem{LE}{Lemma}
\newtheorem{CON}{Corollary}
\newtheorem{DEF}{Definition}
\def\BCON{\begin{CON}\label}
\newtheorem{EX}{Example}

\def\BCON{\begin{CON}\label}
\def\ECON{\end{CON}}
\def\BE{\begin{equation}\label}
\def\EE{\end{equation}}
\def\BLE{\begin{LE}\label}
\def\ELE{\end{LE}}

\def\BDEF{\begin{DEF}\label}
\def\EDEF{\end{DEF}}
\def\BEX{\begin{EX}\label}
\def\EEX{\end{EX}}

\catcode`@=11
\@addtoreset{equation}{section}
\catcode`@=12

\def\v2{\par\vspace{2mm}\par}

\input{tcilatex}

\begin{document}
\title{*-Balanced Fuzzy Graphs}
\author{Talal AL-Hawary}
\address{Department of Mathematics\\
Yarmouk University\\
Irbid-Jordan\\
talalhawary@yahoo.com}
\author{Laith AlMomani}
\address{Mathematics Department\\
Irbid National University\\
Irbid-Jordan\\
thelionlaith1@gmail.com}
\maketitle

\begin{abstract}
Our aim in this paper is to introduce the relatively new concept of
*-density of a fuzzy graph and *-balanced fuzzy graph. Several examples and
results are also provided. In addition, many operations on fuzzy graphs that
preserves *-balanced are explored.
\end{abstract}

\renewcommand{\baselinestretch}{1} \renewcommand{\baselinestretch}{1} 
\vspace{2mm}

\renewcommand{\baselinestretch}{1.45}

\noindent \hspace{-0.2in}\textbf{Keywords: }Fuzzy graph. *-density ,
*-balanced.

\noindent \hspace{-0.2in}\textbf{AMS Subject Classification. }05C72.

\section{Introduction}

Graph theory has several interesting applications in system analysis,
operations research, economics and many other fields. Since most of the time
the aspects of graph and graph problems are uncertain, it is a good idea to
deal with these aspects via the methods of fuzzy logic. The notion of fuzzy
set was first introduced by Zadeh \cite{zad} in his landmark paper "Fuzzy
sets" in 1965 and the concept of fuzzy graph was first introduced by
Rosenfeld \cite{ros} in his paper "Fuzzy Graphs". Since that time, several
authors explored this type of graphs. As the notions of degree, complement,
completeness, regularity and many others play very important role in the
crisp graph case, it is a nice idea to try to see what corresponds to these
notions in the case of fuzzy graphs.

Sunitha and Kumar \cite{sun} defined several new operations on fuzzy graphs
and they also modified the definition of complement of a fuzzy graph so that
to agree with the crisp case in graph. In 2011, AL-Hawary \cite{tal1}
introduced the new concept of balanced fuzzy graphs. He defined three new
operations on fuzzy graphs and explored what classes of fuzzy graphs are
balanced. Since then, many authors have studied the idea of balanced on
distinct types of fuzzy graphs, see for example \cite{tal25,sun}. Moreover,
Al-Hawary and others explored the idea of balanced fuzzy graphs in \cite%
{tal1,tal2,tal25,tal3,tal4}. We start by recalling some necessary
definitions and results.

\begin{definition}
\cite{sun}A fuzzy subset of a set $V$ is mapping $\ \sigma :V\rightarrow
\lbrack 0,1].$For any $v\in V$, $\sigma (v)$ is called the degree of
membership of $v$ in $\sigma .$ 
\end{definition}

\begin{definition}
\cite{sun}A fuzzy relation on a set $V$ is mapping $\ \mu :V\times
V\rightarrow \lbrack 0,1$. A fuzzy relation $\mu $ on a fuzzy subset  $%
\sigma $ is a fuzzy relation on $V$ such that $\mu (u,v)\leq \sigma
(u)\wedge \sigma (v)$ for all $u,v\in V$, where $\wedge $ stands for minimum.
\end{definition}

\begin{definition}
\cite{sun}A fuzzy graph is a pair of functions $G:(\sigma ,\mu )$ where $\
\sigma $ is a fuzzy subset of  $V$ and $\mu $ is a symmetric fuzzy relation
on $\ \sigma $. The underlying crisp graph of \ $G$ 
is denoted by $G^{\ast }:(\sigma ^{\ast },\mu ^{\ast })$ where $\sigma
^{\ast }$ is referred to as the nonempty subset $V$ \ of nodes and $\ \mu
^{\ast }=E\subseteq V\times V.$
\end{definition}

\qquad All through this paper, we only consider non-empty fuzzy graphs.

\begin{definition}
\cite{nag3}. A fuzzy graph $G:(\sigma ,\mu )$ is called \emph{complete} if $\mu
(u,v)=\sigma (u)\wedge \sigma (v)$ for all $u,v\in V$ and . A fuzzy graph $G:(\sigma ,\mu )$ is called \emph{\ strong}
if $\mu (u,v)=\sigma (u)\wedge \sigma (v)$ for all $u,v\in E $ .
\end{definition}

Note that any complete fuzzy graph is strong, but the converse needs not be
true. Let $G_{1}:(\sigma _{1},\mu _{1})$ and $G_{2}:(\sigma _{2},\mu _{2})$
be two fuzzy graphs.

\begin{definition}
\cite{bhu}. Two fuzzy graphs $G_{1}$and $G_{2}$ are \emph{%
isomorphic} if there exists a bijection $h:V_{1}\rightarrow V_{2}$ such that 
$\sigma _{1}(x)=\sigma _{2}(h(x))$ and $\mu _{1}(x,y)=\mu _{2}(h(x),h(y))$
for all $x,y\in V_{1}.$
\end{definition}

\begin{lemma}
\label{l0}Let $G_{1}:(\sigma _{1},\mu _{1})$ and $G_{2}:(\sigma _{2},\mu _{2})$ be
isomorphic fuzzy grapgs. Then $\sum_{v\in V_{1}}\sigma _{1}(v)=$ $\sum_{v\in
V_{2}}\sigma _{2}(v)$ and $\sum_{u,v\in V_{1}}\mu _{1}(u,v)=$ $\sum_{u,v\in
V_{2}}\mu _{2}(u,v).$
\end{lemma}

Several operations on fuzzy graphs were introduced in \cite{sun} such as
union $G_{1}\cup G_{2}$, the join $G_{1}+G_{2},$the Cartesian product $%
G_{1}\times G_{2}$ and the composition $G_{1}\circ G_{2}.$Also recently in 
\cite{tal1}, the operations of direct product $G_{1}\sqcap G_{2},$
semi-direct product $G_{1}\bullet G_{2}$ and strong product $G_{1}\otimes
G_{2}$. In addition, both authors studied the operations that preserves
balanced notion. For more on operations on fuzzy graphs, see \cite%
{tal1,tal2,tal3,tal4,sun}.

\begin{definition}
\cite{sun}. The \emph{complement} of a fuzzy graph $G:(\sigma ,\mu )$ is a
fuzzy graph $G^{^{c}}:(\sigma ^{^{c}},\mu ^{^{c}})$, where $\sigma
^{c}=\sigma $ and $\mu ^{c}(u,v)=\sigma (u)\wedge \sigma (v)-\mu (u,v),$ $%
\forall u,v\in V.$\emph{\medskip }
\end{definition}

Next we recall the following two results from \cite{sun}.

\begin{lemma}
\label{l1}Let $G:(\sigma ,\mu )$ be a self-complemetary fuzzy graph. Then $%
\sum_{u,v\in V}\mu (u,v)=(1/2)\sum_{u,v\in V}(\sigma (u)\wedge \sigma (v))$
\end{lemma}

\begin{lemma}
\label{l2}Let $G:(\sigma ,\mu )$ be a fuzzy graph with $\mu (u,v)=(1/2)(\sigma
(u)\wedge \sigma (v))$ for all $u,v\in V.$Then $G$ is self-complemetary.
\end{lemma}

\begin{definition}
\cite{nag3} Let $G:(\sigma ,\mu )$ be a fuzzy graph. The degree of a vertex $u$ is $%
d_{G}(u)=\sum_{u\neq v}\mu (u,v)$. The total degree of vertex $u$ is $%
td_{G}(u)=\sigma (u)+d_{G}(u).$ 
\end{definition}

\begin{definition}
\cite{nag3} A fuzzy graph $G:(\sigma ,\mu )$ is called k-regular if $d_{G}(u)=k$ for
every $u\in V.$ $G:(\sigma ,\mu )$ is called k-totally regular if $%
td_{G}(u)=k$ for every $u\in V.$
\end{definition}

In general there does not exist any relationship between regular fuzzy
graphs and totally fuzzy graphs.

\begin{definition}
\cite{nag3}. Let $G:(\sigma ,\mu )$ be a fuzzy graph. Then $\sigma $ is
called a \emph{c-constant function} if $\sigma (v)=c$ for all $v\in V$ and $%
\mu $ is called a \emph{c-constant function} if $\mu (u,v)=c$ for all $%
u,v\in V$.
\end{definition}

Our aim in this paper is to define the concept of *-density of a fuzzy
graph. In fact, it is a modification of the concept of density of fuzzy
graph in which we change the denominator so as to satisfy more properties
and to agree more with what is know about density of graphs. Moreover, we
introduce and explore what we call *-balanced fuzzy graph. Several examples
and results are also provided and certain classes of *-balanced fuzzy graphs
are given.

\section{*-Balanced Fuzzy Graphs}

\qquad The main idea in this section is to define the concept of *-density
of a fuzzy graph and *-balanced fuzzy graph. We explore these notions and we
get some nice results that are analogous to those in \cite{tal1}. We begin
by the following Definition:

\begin{definition}
The \emph{*-density} of a fuzzy graph $G:(\sigma ,\mu )$ is $D^{\ast }(G)=%
\frac{2\sum_{u,v\in V}\mu (u,v)}{\sum_{u\in V}\sigma (u)}$. G is \emph{%
*-balanced} if $D^{\ast }(H)\leq D^{\ast }(G)$ for all non-empty fuzzy
subgraphs $H$ of $G$ .
\end{definition}

\begin{theorem}
Any complete fuzzy graph with $|V|\geq |E|$ has *-density $D^{\ast }(G)\leq 2
$.
\end{theorem}

\begin{proof}
Let $G$ be complete fuzzy graph. Since $\sum_{u\in V}\sigma (u)\geq
\sum_{u,v\in V}\mu (u,v)$ , then $\frac{\sum_{u,v\in V}\mu (u,v)}{\sum_{u\in
V}\sigma (u)}\leq 1$ and hence $D^{\ast }(G)\leq 2$.
\end{proof}

\begin{theorem}
\label{th1}Every self-complementary fuzzy graph has a density less than or
equal 1.
\end{theorem}

\begin{proof}
Let $G$ be self-complementary fuzzy graph. Then as $\sum_{u,v\in V}\sigma
(u)\wedge \sigma (v)\leq \sum_{u\in V}\sigma (u)$, $D^{\ast }(G)\leq \frac{%
2\sum_{u,v\in V}\mu (u,v)}{\sum_{u,v\in V}\sigma (u)\wedge \sigma (v)}$. Now
by Lemma \ref{l2}, $D^{\ast }(G)\leq \frac{2\sum_{u,v\in V}\mu (u,v)}{%
2\sum_{u,v\in V}\mu (u,v)}=1$.
\end{proof}

The converse of the preceding result needs not true.

\begin{theorem}
Let $G:(\sigma ,\mu )$ be fuzzy graph such that $\mu (u,v)=\frac{\sigma
(u)\wedge \sigma (v)}{2}$ for all $u,v\in V.$Then $D^{\ast }(G)\leq 1$.
\end{theorem}

\begin{proof}
By Lemma \ref{l1}, $G$ is self-complementary and thus by Theorem \ref{th1}, $%
D^{\ast }(G)\leq 1$.
\end{proof}

\begin{lemma}
\label{lem}Let $G$ $_{1}$ and $G$ $_{2}$ be complete fuzzy graphs. Then $D^{\ast
}(G_{i})\leq D^{\ast }(G_{1}\sqcap G_{2})$ for $i=1,2$ if and only if $%
D^{\ast }(G_{1})=D^{\ast }(G_{2})=D^{\ast }(G_{1}\sqcap G_{2}).$
\end{lemma}

\begin{proof}
If $D^{\ast }(G_{i})\leq D^{\ast }(G_{1}\sqcap G_{2})$ for $i=1,2,$ then 
\begin{eqnarray*}
D^{\ast }(G_{1}) &=&2(\sum_{u_{1},u_{2}\in V_{1}}\mu
_{1}(u_{1},u_{2}))/\sum_{u_{1}\in V_{1}}\sigma _{1}(u_{1}) \\
&\geq &2(\sum_{\substack{ u_{1},u_{2}\in V_{1}  \\ v_{1},v_{2}\in V_{2}}}\mu
_{1}(u_{1},u_{2})\wedge \sigma _{2}(v_{1})\wedge \sigma _{2}(v_{2}))/(\sum 
_{\substack{ u_{1},u_{2}\in V_{1}  \\ v_{1},v_{2}\in V_{2}}}(\sigma
_{1}(u_{1})\wedge \sigma _{2}(v_{1})\wedge \sigma _{2}(v_{2})) \\
&\geq &2(\sum_{\substack{ u_{1},u_{2}\in V_{1}  \\ v_{1},v_{2}\in V_{2}}}\mu
_{1}(u_{1},u_{2})\wedge \mu _{2}(v_{1},v_{2}))/(\sum_{\substack{ %
u_{1},u_{2}\in V_{1}  \\ v_{1},v_{2}\in V_{2}}}(\sigma _{1}(u_{1})\wedge
\sigma _{2}(v_{1})) \\
&=&2(\sum_{\substack{ u_{1},u_{2}\in V_{1}  \\ v_{1},v_{2}\in V_{2}}}\mu
_{1}\sqcap \mu _{2}((u_{1},u_{2})(v_{1},v_{2}))/(\sum_{\substack{ %
u_{1},u_{2}\in V_{1}  \\ v_{1},v_{2}\in V_{2}}}(\sigma _{1}\sqcap \sigma
_{2}((u_{1},u_{2})(v_{1},v_{2}))) \\
&=&D^{\ast }(G_{1}\sqcap G_{2}).
\end{eqnarray*}%
Hence $D^{\ast }(G_{1})\geq D^{\ast }(G_{1}\sqcap G_{2})$ and thus $D^{\ast
}(G_{1})=D^{\ast }(G_{1}\sqcap G_{2}).$ Similarly, $D^{\ast }(G_{2})=D^{\ast
}(G_{1}\sqcap G_{2}).$ Therefore, $D^{\ast }(G_{1})=D^{\ast }(G_{2})=D^{\ast
}(G_{1}\sqcap G_{2}).$

The converse is trivial.
\end{proof}

\begin{theorem}
\label{th}Let $G_{1}$ and $G_{2}$ be fuzzy complete *-balanced graphs. Then $%
G_{1}\sqcap G_{2}$ is *-balanced if and only if $D^{\ast }(G_{1})=D^{\ast }(G_{2})=D^{\ast }(G_{1}%
\sqcap G_{2}).$
\end{theorem}

\begin{proof}
If $G$ $_{1}\sqcap G_{2}$ is *-balanced, then $D^{\ast }(G_{i})\leq D^{\ast
}(G_{1}\sqcap G_{2})$ for $i=1,2$ and by Lemma \ref{lem}, $D^{\ast
}(G_{1})=D^{\ast }(G_{2})=D^{\ast }(G_{1}\sqcap G_{2}).$

Conversely, if $D^{\ast }(G_{1})=D^{\ast }(G_{2})=D^{\ast }(G_{1}\sqcap
G_{2})$ and $H$ is a fuzzy subgraph of $G$ $_{1}\sqcap G_{2}$, then there
exist fuzzy subgraphs $H_{1}$ of $G$ $_{1}$ and $H_{2}$ of $G$ $_{2}.$ As $G$
$_{1}$ and $G$ $_{2}$ are *-balanced and $D^{\ast }(G_{1})=D^{\ast
}(G_{2})=n_{1}/r_{1},$ then $D^{\ast }(H_{1})=a_{1}/b_{1}\leq n_{1}/r_{1}$
and $D^{\ast }(H_{2})=a_{2}/b_{2}\leq n_{1}/r_{1}.$ Thus $%
a_{1}r_{1}+a_{2}r_{1}\leq b_{1}n_{1}+b_{2}n_{1}$ and hence $D^{\ast }(H)\leq
(a_{1}+a_{2})/(b_{1}+b_{2})\leq n_{1}/r_{1}=D^{\ast }(G_{1}\sqcap G_{2}).$
Therefore, $G$ $_{1}\sqcap G_{2}$ is *-balanced.
\end{proof}

The above result needs not be true when one of the fuzzy graphs is not
complete.

The preceding result needs not be true if the operation $\sqcap $ is
replaced by $\bullet $, $\otimes $, $+$, $\circ $, $\times $. We only give
an example of the case $\circ $. We end this section by showing that
isomorphism between fuzzy graphs preserve *-balanced.

\begin{theorem}
Let $G$ $_{1}$ and $G$ $_{2}$ be isomorphic fuzzy graphs. If $G$ $_{2}$ is
*-balanced, then $G$ $_{1}$ is *-balanced.
\end{theorem}

\begin{proof}
Let\ $h:V_{1}\rightarrow V_{2}$ be a bijection such that $\sigma
_{1}(x)=\sigma _{2}(h(x))$ and $\mu _{1}(x,y)=\mu _{2}(h(x),h(y))$ for all $%
x,y\in V_{1}.$ By Lemma \ref{l0}, $\sum_{x\in V_{1}}\sigma
_{1}(x)=\sum_{x\in V_{2}}\sigma _{2}(x)$ and $\sum_{x,y\in V_{1}}\mu
_{1}(x,y)=\sum_{x,y\in V_{2}}\mu _{2}(x,y).$ If $H_{1}=(\sigma
_{1}^{^{\prime }},\mu _{1}^{^{\prime }})$ is a fuzzy subgraph of $G$ $_{1}$
with underlying set $W,$ then $H_{2}=(\sigma _{2}^{^{\prime }},\mu
_{2}^{^{\prime }})$ is a fuzzy subgraph of $G$ $_{2}$ with underlying set $%
h(W)$ where $\sigma _{2}^{^{\prime }}(h(x))=\sigma _{1}^{^{\prime }}(x)$ and 
$\mu _{2}^{^{\prime }}(h(x),h(y))=\mu _{1}^{^{\prime }}(x,y)$ for all $%
x,y\in W.$ Since $G$ $_{2}$ is *-balanced, $D^{\ast }(H_{2})\leq D^{\ast
}(G_{2})$ and so%
\begin{equation*}
2(\sum_{x,y\in W}\mu _{2}^{^{\prime }}(h(x),h(y)))/\sum_{x\in W}\sigma
_{2}^{^{\prime }}(x)\leq 2(\sum_{x,y\in V_{2}}\mu _{2}(x,y))/\sum_{x\in
W}\sigma _{2}^{^{\prime }}(x)
\end{equation*}%
and so 
\begin{equation*}
2(\sum_{x,y\in W}\mu _{1}(x,y))/\sum_{x\in W}\sigma _{1}^{^{\prime }}(x)\leq
2(\sum_{x,y\in V_{1}}\mu _{1}(x,y))/\sum_{x\in W}\sigma _{1}^{^{\prime }}(x).
\end{equation*}%
Thus $D^{\ast }(H_{1})\leq D^{\ast }(G_{1}).$ Therefore, $G$ $_{1\text{ }}$%
is *-balanced.
\end{proof}

\section{On regular fuzzy graphs}

\begin{theorem}
If $G:(\sigma ,\mu )$ is an r- regular fuzzy graph with $|V|=p$, then $G$
has a density $D^{\ast }(G)=pr/\sum_{v\in V}\sigma (v)$.
\end{theorem}

\begin{proof}
Since $G$ is an r- regular fuzzy graph, then $d_{G}(v)=r$ for all $v\in V$.
Now as $\sum_{v\in V}d_{G}(v)=2\sum_{u,v\in V}\mu (v,u)$, $\sum_{u,v\in
V}\mu (v,u)=\frac{\sum_{v\in V}r}{2}=\frac{pr}{2}$. Thus $D^{\ast
}(G)=pr/\sum_{v\in V}\sigma (v)$.
\end{proof}

\begin{corollary}
If $G:(\sigma ,\mu )$ is an r- regular and $\sigma $ is c- constant
function, then $D^{\ast }(G)=r/c$.
\end{corollary}

\begin{theorem}
If $G:(\sigma ,\mu )$ is an r-totally regular fuzzy graph with $|V|=p$, then 
$G$ has a *-density $D^{\ast }(G)=(pr/\sum_{v\in V}\sigma (v))-1$.
\end{theorem}

\begin{proof}
Since $G$ is an r-totally regular fuzzy graph then $r=td_{G}(u)=d_{G}(u)+%
\sigma (u)$ for all $u\in V$. Thus $\sum_{v\in V}r=\sum_{v\in
V}d_{G}(v)+\sum_{v\in V}\sigma (v)$. Hence $pr=2\sum_{u,v\in V}\mu
(v,u)+\sum_{v\in V}\sigma (v)$ . So $pr/\sum_{v\in V}\sigma (v)=\frac{%
2\sum_{u,v\in V}\mu (v,u)}{\sum_{v\in V}\sigma (v)}+1$. Therefor, $D^{\ast
}(G)=(pr/\sum_{v\in V}\sigma (v))-1$.
\end{proof}

\begin{corollary}
If $G:(\sigma ,\mu )$ is an r-totally regular and $\sigma $ is c- constant
function, then $D^{\ast }(G)=\frac{r}{c}-1$.
\end{corollary}

Note that the operations of $\sqcap $, $\bullet $, $\otimes $, $+$, $\circ $%
, $\times $ do not preserve r-totally regular property. We only give a
counter example for case $\circ $:

\section{Classes of *-Balanced Fuzzy Graphs}

\begin{theorem}
If the complete graph on $n$-vertices $K_{n}$ has $\delta $ as a $c$-
constant function and complete, then $K_{n}$ is *-balanced.
\end{theorem}

\begin{proof}
Now $D^{\ast }(K_{n})=\frac{2(\frac{cn(n-1))}{2}}{cn}=n-1$. Any subgraph $H$
of $K_{n}$ has edges less than $K_{n}$ or less edges and less vertices than $%
K_{n}$. If $H$ has less edges, it is clear that $D^{\ast }(H)\leq D^{\ast
}(K_{n})$. Now if $H$ has less edges and less vertices, say $H$ has $n-s$
vertices, then 
\begin{eqnarray*}
|E(H)| &=&\frac{n(n-1)}{2}-((n-1)+(n-2)+...+(n-s) \\
&=&\frac{n(n-1)}{2}-(sn-\frac{s(s+1)}{2}) \\
&=&\frac{n(n-1)-2sn+s(S+1)}{2}.
\end{eqnarray*}%
Thus 
\begin{eqnarray*}
D^{\ast }(H) &=&\frac{2c(\frac{n(n-1)-2sn+s(S+1)}{2})}{c(n-s)} \\
&=&\frac{n^{2}-n-2sn+s^{2}+s}{n-s} \\
&=&n-(s+1).
\end{eqnarray*}%
As $1<s+1$, $D^{\ast }(H)\leq n-1=D^{\ast }(K_{n})$ and so $K_{n}$ is
*-balanced.
\end{proof}

Even when $\mu $ is not a constant function but $\sigma $ is a constant
function, $K_{n}$ needs not be *-balanced as shown in Figure 6. Also when $%
\sigma $ is not a constant function but $\mu $ is a constant function, $K_{n}
$ needs not be *-balanced.

\begin{theorem}
If \ the cycle $C_{n}$ has $\sigma $ as a c- constant function and strong
for $n>3$, then $C_{n}$ is *-balanced.
\end{theorem}

\begin{proof}
Now $D^{\ast }(C_{n})=\frac{2cn}{cn}=2$. Any subgraph $H$ of has edges less
than $C_{n}$ or less edges and less vertices than $C_{n}$. If $H$ has less
edges, it is clear that $D^{\ast }(H)\leq D^{\ast }(C_{n})$. If $H$ has less
edges and less vertices than $C_{n}$, say $H$ has $n-s$ vertices, then we
have three cases:

\begin{enumerate}
\item[Case 1.] No two of the s-vertices are adjacent. Then $|E(H)|=n-2s$ and
so $D^{\ast }(H)=\frac{2c(n-2c)}{c(n-c)}=\frac{2(n-2c)}{n-c}\leq 2.$

\item[Case 2.] The subgraph consisting of these s vertices is isomorphic to
a path graph. Then $|E(H)|=n-(2+s-1)=n-s-1$. Hence, $D^{\ast }(H)=\frac{%
2c(n-s-1)}{c(n-s)}\leq 2.$

\item[Case3.] The subgraph consisting of these $s$ vertices has $s_{1}$
vertices of those in Case 1 and $s_{2}$ vertices of those in Case 2. Then $%
|E(H)|=n-2s_{1}-(2+s_{2}-1)=$ $n-2s_{1}-s_{2}-1.$ Hence $D^{\ast }(H)=\frac{%
2c(n-2s_{1}-s_{2}-1)}{c(n-s_{1}-s_{2})}\leq 2.$

Therefore $C_{n}$ is *-balanced.
\end{enumerate}
\end{proof}

\begin{theorem}
If the Petersen fuzzy graph $P_{10}$ has $\sigma $ as a c- constant function
and strong, then $P_{10}$ is *-balanced.
\end{theorem}

\begin{proof}
Now $D^{\ast }(P_{10})=\frac{2c(15)}{10c}=3$. Any subgraph $H$ of $P_{10}$
has edges less than $P_{10}$ or less edges and less vertices than $P_{10}$.
If $H$ has less edges, it is clear that $D^{\ast }(H)\leq D^{\ast }(P_{10})$%
. If $H$ has less edges and less vertices than $P_{10}$, say $H$ has $10-s$
vertices, then we have three cases:

\begin{enumerate}
\item[Case1. ] No two of the s-vertices are adjacent. Then $|E(H)|=15-3c$
and as $D^{\ast }(H)=\frac{2c(15-3s)}{c(10-s)}\leq 3.$

\item[Case2.] The subgraph consisting of these s vertices is isomorphic to a
path graph. Then $|E(H)|=15-(3+2(s-1))=14-2s$ . Hence, $D^{\ast }(H)=\frac{%
2c(14-2s)}{c(10-s)}\leq 3$ for all $s>1$ and hence .

\item[Case3.] The subgraph consisting of these $s$ vertices has $s_{1}$
vertices of those in Case 1 and $s_{2}$ vertices of those in Case 2. Then $%
|E(H)|=15-3s_{1}-(3+2(s_{2}-1))=$ $14-3s_{1}-2s_{2}.$ Hence $D^{\ast }(H)=%
\frac{2c(14-3s_{1}-2s_{2})}{c(10-s)}=\frac{3(\frac{28}{3}-2s_{1}-\frac{4}{3}%
s_{2})}{10-s_{1}-s_{2}}\leq 3.$

Therefore $P_{10}$ is *-balanced.
\end{enumerate}
\end{proof}

\begin{theorem}
If $K_{n,n}$\ has $\sigma $ as a c- constant function and strong, then $%
K_{n,n}$ is *-balanced.
\end{theorem}

\begin{proof}
Now $D^{\ast }(K_{n,n})=\frac{2c(n^{2})}{2nc}=n$. Any subgraph $H$ of $%
K_{n,n}$ has edges less than $K_{n,n}$ or less edges and less vertices than $%
K_{n,n}$. If $H$ has less edges, it is clear that $D^{\ast }(H)\leq D^{\ast
}(K_{n,n})$. If $H$ has less edges and less vertices than $K_{n,n}$, say $H$
has $2n-s$ vertices, then we have three cases:

\begin{enumerate}
\item[Case1. ] No two of the s-vertices are adjacent. Then $D^{\ast
}(K_{n-s,n})=\frac{2c(n(n-s)}{c(2n-s)}=\frac{n(n-s)}{n-s/2}\leq n.$

\item[Case2.] The subgraph is $K_{n-s_{1},n-s_{2}}$ where $s_{1}+s_{2}=s$.
Then $D^{\ast }(H)=\frac{2c((n-s_{1})(n-s_{2}))}{c(2n-s)}=\frac{%
n(n-s)-s_{1}s_{2}}{n-s/2}\leq n.$

Therefore $K_{n,n}$ is *-balanced.
\end{enumerate}
\end{proof}

\section{Conclusion}

In this paper, we defined the concept of *-density of a fuzzy graph and we
introduced and explored what we call *-balanced fuzzy graph. Several
examples and results were also provided and certain classes of *-balanced
fuzzy graphs are given. In addition, *-balanced fuzzy graphs were discussed.

\renewcommand{\baselinestretch}{1} \vspace*{10mm} 

\end{document}